\documentclass[12pt]{article}

\usepackage[T1]{fontenc}
\usepackage[a4paper]{geometry}
\usepackage{subfigure}
\usepackage[ruled,linesnumbered]{algorithm2e}
\usepackage{amsmath}
\usepackage{amsthm}
\usepackage{amssymb}
\usepackage[colorlinks=true,linkcolor=blue,citecolor=blue]{hyperref}
\usepackage[capitalise]{cleveref}
\usepackage{graphicx}
\usepackage{paralist}
\usepackage{todonotes}
\usepackage{enumitem}
\usepackage{caption}
\newcommand{\R}{\mathbb{R}}

\newcommand{\N}{\mathbb{N}}

\newcommand{\set}[1]{\left\{#1\right\}}

\newcommand{\norms}[1]{\Vert#1\Vert}
\newcommand{\iprod}[1]{\left\langle#1\right\rangle}
\newcommand{\iprods}[1]{\langle#1\rangle}

\newcommand{\dom}[1]{\mathrm{dom}\, #1}

\newcommand{\ri}{\operatorname{ri}}

\newtheorem{fact}{Fact}[section]
\theoremstyle{definition}
\newtheorem{definition}{Definition}[section]

\newtheorem{proposition}{Proposition}[section]

\theoremstyle{remark}

\newtheorem{example}{Example}[section]

\newtheorem{assumption}{Assumption}

\crefname{assumption}{Assumption}{Assumptions}
\Crefname{enumi}{}{}
\setlist[enumerate,1]{label=(\roman*)}



\graphicspath{{Figures/}}

\makeatother

\graphicspath{{Figures/}}

\begin{document}

\title{Using positive spanning sets to achieve d-stationarity with the Boosted DC Algorithm}

\author{F. J. Arag\'on Artacho$^{*}$
\and
	R. Campoy\footnote{Department of Mathematics, University of Alicante, Alicante, Spain. \newline\indent{~~} Email: \texttt{francisco.aragon@ua.es} and \texttt{ruben.campoy@ua.es}}
\and	
 P. T. Vuong\footnote{Department of Mathematics, University of Vienna, Austria and School of Mathamtical Sciences, University of Southampton, SO17 1BJ, Southampton, UK. \newline\indent{~~} Email: \texttt{t.v.phan@soton.ac.uk}}
 }

\date{\today\\ \bigskip \emph{This paper is dedicated to Professor Marco A. L\'opez Cerd\'a\\ \vspace{3pt} on the occasion of his 70th birthday}}

\maketitle

\begin{abstract}
The Difference of Convex functions Algorithm (DCA) is widely used for minimizing the difference of two convex functions. A recently proposed accelerated version, termed BDCA for Boosted DC Algorithm, incorporates a line search step to achieve a larger decrease of the objective value at each iteration. Thanks to this step, BDCA usually converges much faster than DCA in practice. The solutions found by DCA are guaranteed to be critical points of the problem, but these may not be local minima. Although BDCA tends to improve the objective value of the solutions it finds, these are frequently just critical points as well. In this paper we combine BDCA with a simple Derivative-Free Optimization (DFO) algorithm to force the d-stationarity (lack of descent direction) at the point obtained. The potential of this approach is illustrated through some computational experiments on a Minimum-Sum-of-Squares clustering problem. Our numerical results demonstrate that the new method provides better solutions while still remains faster than DCA in the majority of test cases.
\end{abstract}


\section{Introduction}\label{sec:intro}
In this paper, we are interested in solving the following unconstrained DC (difference of convex functions) optimization problem:
\begin{equation}\label{eq:P}\tag{$\mathcal{P}$}
\displaystyle\min_{x\in\R^m} \left\{\phi(x) := g(x) - h(x)\right\}
\end{equation}
where $g:\R^m\to\R\cup\{+\infty\}$ and $h:\R^m\to\R\cup\{+\infty\}$ are proper, closed and convex functions, and $g$ is smooth, with the conventions:
\begin{gather*}
(+\infty)-(+\infty)=+\infty,\\
(+\infty)-\lambda=+\infty\quad\text{and}\quad\lambda-(+\infty)=-\infty,\quad\forall\lambda\in{]-\infty,+\infty[}.
\end{gather*}

Problem~\eqref{eq:P} can be tackled by the well-known \emph{DC algorithm (DCA)}~\cite{TT97,DC86}. DC programming has become an active research field for the last few decades~\cite{An2018} and DCA has been successfully applied to many real-world problems arising in different fields (see, e.g.,~\cite{tao2005dc}). Although DCA performs well in practice, its convergence can be fairly slow for some particular problems. In order to speed up the scheme, an accelerated version of the algorithm, called \emph{Boosted DC algorithm (BDCA)}, has been recently proposed in~\cite{BDCA2018,nBDCA}. The BDCA performs a line search at the point generated by the classical DCA, which allows to achieve a larger decrease in the objective value at each iteration. In the numerical experiments reported in~\cite{BDCA2018,nBDCA} it was shown that BDCA was not only faster than DCA, but also often found solutions with lower objective value. However, although both algorithms are proved to converge to critical points of~\eqref{eq:P}, there is no guarantee that these points are local minima. For this reason, a simple trick to achieve better solutions consists in running the algorithms from different starting points. Another approach has been recently used in~\cite{Welington}, where the authors incorporated an inertial term into the algorithm making it converge to \emph{better} critical points. In the recent work~\cite{WelingtonJogo}, the author proposed a DC scheme which is able to compute d-stationary points. Although this algorithm permits to address problems where the function $g$ is nonsmooth, the second component function $h$ needs to be the pointwise maximum of finitely many differentiable functions.

The aim of this paper is to show that it is possible to combine BDCA with a simple DFO (Derivative-Free Optimization) routine to guarantee {d-stationarity}  at the limit point obtained by the algorithm.
As a representative application, we perform a set of numerical experiments on the Minimum Sum-of-Squares Clustering problem studied in~\cite{nBDCA} to illustrate this observation.
This problem has many critical points, where both DCA and BDCA tend to easily get trapped in.
As a byproduct of the DFO step, we observe that in some problems a single run of the new algorithm is able to provide better solutions than those obtained by multiple restarts of DCA.

The rest of this paper is organized as follows. In Section~\ref{sec:pre_results} we recall some preliminary results. We propose a new variant of BDCA, named BDCA$+$, in Section~\ref{sec:BDCA}. The results of some numerical experiments are presented in Section~\ref{sec:application1}, where we compare the performance of DCA, BDCA and BDCA$+$ on several test cases. We finish with some conclusions in Section~\ref{sec:concl}.

\section{Preliminaries}\label{sec:pre_results}

Throughout this paper, $\langle x, y\rangle$ denotes the inner product of $x,y\in\R^m$, and $\norms{\cdot}$ corresponds to the induced norm given by $\norms{x}=\sqrt{\langle x,x\rangle}$. For any extended real-valued function $f:\R^m\to\R\cup\{+\infty\}$, the set $\dom{f} := \set{x\in\R^m \mid f(x) < +\infty}$ denotes the (effective) \emph{domain} of $f$. A function $f$ is \emph{proper} if its domain is nonempty. The function $f$ is \emph{coercive} if $f(x)\to +\infty$ whenever $\norms{x}\to +\infty$, and it is said to be \emph{convex}~if
$$f(\lambda x+(1-\lambda)y)\leq \lambda f(x) +(1-\lambda)f(y), \quad \text{for all } x,y\in\R^m \text{ and } \lambda\in[0,1].$$
Further, $f$ is \emph{strongly convex} with strong convexity parameter $\rho>0$ if $f-\frac{\rho}{2}\norms{\cdot}^2$ is convex, i.e., when
$$f(\lambda x+(1-\lambda)y)\leq \lambda f(x) +(1-\lambda)f(y)-\frac{\rho}{2}\lambda(1-\lambda)\norms{x-y}^2,$$
for all $x,y\in\R^m$ and $\lambda\in[0,1]$. For any convex function $f$, the \emph{subdifferential} of $f$ at $x\in\R^m$ is the set
$$\partial{f}(x) := \set{ w\in\R^m \mid f(y) \geq f(x) + \iprod{w,y-x},~\forall y\in\R^m}.$$
If $f$ is differentiable at $x$, then $\partial f(x)=\{\nabla f(x)\}$, where $\nabla f(x)$ denotes the \emph{gradient} of $f$ at $x$. The one-sided \emph{directional derivative} of $f$ at $x$ with respect to the direction $d\in\R^m$ is defined by
$$f'(x;d):=\lim_{t\searrow 0}\frac{f(x+td)-f(x)}{t}.$$

Before going to the main contribution of this paper in Section~\ref{sec:BDCA}, we state our assumptions imposed on \eqref{eq:P}. We also recall some preliminary notions and basic results  which will be used in the sequel.

\subsection{Basic Assumptions}

\begin{assumption}\label{as:A1}
Both functions $g$ and $h$  in \eqref{eq:P} are strongly convex on their domain for the same strong convexity parameter $\rho >0$.
\end{assumption}

\begin{assumption}\label{as:A2}
	The function $h$ is subdifferentiable at every point in $\dom{h}$; that is, $\partial h(x)\neq\emptyset$ for all $x\in\dom{h}$.
\end{assumption}

\begin{assumption}\label{as:A3}	The function $g$ is continuously differentiable on an open set containing $\dom{h}$  and
	\begin{equation*}\label{eq:phi_bounded_below}
	\phi^{\star} := \inf_{x\in\R^{m}}\phi(x) > -\infty.
	\end{equation*}
\end{assumption}

\Cref{as:A1} is not restrictive, as one can always rewrite the objective function as $\phi = (g+q) - (h+q)$ for any strongly convex function $q$ (e.g., $q=\frac{\rho}{2}\norms{\cdot}^2$). Observe that \Cref{as:A2} holds for all $x\in\ri\dom{h}$ (by \cite[Theorem~23.4]{RT}). A key property for our method is the smoothness of $g$ in \Cref{as:A3}, which cannot be in general omitted (see~\cite[Example~3.2]{nBDCA}).

\subsection{Optimality Conditions}

Under~\cref{as:A2,as:A3} the following well-known necessary condition for local optimality holds.

\begin{fact}[First-order necessary optimality condition]
If $x^\star\in\dom{\phi}$ is a local minimizer of problem~\eqref{eq:P}, then
\begin{equation}\label{eq:OC}
\partial{h}(x^\star)=\{\nabla{g}(x^\star)\}.
\end{equation}
\end{fact}
\begin{proof}
See {\cite[Theorem~3']{Toland79}}.
\end{proof}

Any point satisfying condition~\eqref{eq:OC} is called a \emph{d(irectional)-stationary point} of \eqref{eq:P}. We say that $x^{\star}$ is a \emph{critical point} of \eqref{eq:P} if
$$\nabla{g}(x^{\star}) \in \partial{h}(x^{\star}).$$
Clearly, d-stationary points are critical points, but the converse is not true in general (see, e.g.,~\cite[Example~1]{BTU16}). In our setting, the notion of critical point coincides with that of \emph{Clarke stationarity}, which requires that zero belongs to the Clarke subdifferential at $x^\star$ (see, e.g.,~\cite[Proposition~2]{GGMB18}). The next result establishes that the d-stationary points of~\eqref{eq:P} are precisely those points for which the directional derivative is zero for every direction.

\begin{proposition}\label{fact:stationary_nodescent}
A point $x^\star\in\dom{\phi}$ is a d-stationary point of~\eqref{eq:P} if and only if
\begin{equation}\label{eq:no_descent}
\phi'(x^\star; d)= 0,\quad\text{for all } d\in \R^m.
\end{equation}
\end{proposition}
\begin{proof}
If $x^\star$ is a d-stationary point of~\eqref{eq:P}, then by~\cite[Theorem~25.1]{RT} we know that $h$ is differentiable at $x^\star$. Therefore, for any $d\in\R^m$, we have
$$\phi'(x^\star; d) = \langle \nabla{g}(x^\star), d \rangle- \langle \nabla{h}(x^\star), d \rangle=0.$$

For the converse implication, pick any $v\in\partial h(x^\star)\neq\emptyset$ (by \Cref{as:A2}) and observe that, for any $d\in\R^m$, we have that
\begin{align*}
\phi'(x^\star; d) & = g'(x^\star;d)-h'(x^\star;d)\\
& = \langle \nabla{g}(x^\star), d \rangle -\lim_{t\searrow 0}\frac{h(x^\star+td)-h(x^\star)}{t}\\
& \leq  \langle \nabla{g}(x^\star)-v, d \rangle .
\end{align*}
Hence, if $x^\star$ satisfies~\eqref{eq:no_descent}, one must have
$$ \langle   \nabla{g}(x^\star)-v, d \rangle \geq 0,\quad \text{for all } d\in\R^m,$$
which is equivalent to $ \nabla{g}(x^\star)-v = 0$.
As $v$ was arbitrarily chosen in $\partial h(x^\star)$, we conclude that $\partial h(x^*)=\{\nabla g(x^\star)\}$.
\end{proof}

\subsection{DCA and Boosted DCA}

In this section, we recall the iterative procedure DCA and its accelerated extension, BDCA, for solving problem~\eqref{eq:P}. The DCA iterates by solving a sequence of approximating convex subproblems, as described next in Algorithm~\ref{alg:DCA}.

\begin{algorithm}[ht!]\caption{DCA (DC Algorithm)}\label{alg:DCA}
\KwIn{An initial point $x_{0}\in \R^m$ and a desired tolerance $\varepsilon\geq 0$;}
\Begin{$k\leftarrow 0$;

Select $u_k \in \partial h(x_k)$ and compute the unique solution $y_k$ of
\begin{equation}\label{eq:sub_prob_DCA}
\tag{$\mathcal{P}_{k}$} ~
\displaystyle\min_{x\in\R^m} \left\{\phi_k(x) := g(x) - \iprods{u_k, x}\right\};
\end{equation}

\uIf{$\|y_{k}-x_k\|\leq\varepsilon$}{
\textbf{stop} and \textbf{return} $y_k$;}
\Else{$x_{k+1}=y_k$;}

$k\leftarrow k+1$ and \textbf{go to} line~3;}
\end{algorithm}

The key feature that makes the DCA work, stated next in~\Cref{fact:DCA}\Cref{fact:DCA_i}, is that the solution of~\eqref{eq:sub_prob_DCA} provides a decrease in the objective value of~\eqref{eq:P} along the iterations. Actually, an analogous result holds for the dual problem, see~\cite[Theorem~3]{TT97}. In~\cite{BDCA2018}, the authors showed that the direction generated by the iterates of DCA, namely $d_k:=y_k-x_k$, provides a descent direction of the objective function at $y_k$ when the functions $g$ and $h$ in~\eqref{eq:P} are assumed to be smooth. This result was later generalized in \cite{nBDCA} to the case where $h$ satisfies~\Cref{as:A2}. The following result collects these properties.

\begin{fact}\label{fact:DCA}
Let $x_k$ and $y_k$ be the sequences generated by~\Cref{alg:DCA} and set $d_k:=y_k-x_k$, for all $k\in\N$. Then the following statements hold:
\begin{enumerate}[label=(\alph*)]
\item $\phi(y_k)\leq \phi(x_k)-\rho\norms{d_k}^2$;\label{fact:DCA_i}
\item $\phi'(y_k;d_k)\leq -\rho\norms{d_k}^2$;\label{fact:DCA_ii}
\item there exists some $\delta_k>0$ such that\label{fact:DCA_iii}
$$\phi(y_k+\lambda d_k)\leq \phi(y_k)-\alpha\lambda^2\norms{d_k}^2,\quad\text{for all } \lambda\in{[0,\delta_k[}.$$
\end{enumerate}
\end{fact}
\begin{proof}
See~\cite[Proposition~3.1]{nBDCA}.
\end{proof}

Thanks to~\Cref{fact:DCA}, once $y_k$ has been found by DCA, one can achieve a larger decrease in the objective value of~\eqref{eq:P} by moving along the descent direction $d_k$. Indeed, observe that
$$\phi(y_k+\lambda d_k)\leq \phi(y_k)-\alpha\lambda^2\norms{d_k}^2\leq \phi(x_k)-(\rho+\alpha\lambda^2)\norms{d_k}^2,\quad\text{for all } \lambda\in{[0,\delta_k[}.$$
This fact is the main idea of the BDCA~\cite{BDCA2018,nBDCA}, whose iteration is described next in~\Cref{alg:BDCA}.

\begin{algorithm}[ht!]\caption{BDCA (Boosted DC Algorithm)}\label{alg:BDCA}
\KwIn{An initial point $x_{0}\in \R^m$ and a desired tolerance $\varepsilon\geq 0$. Choose some parameters $\alpha>0$ and $\beta\in{]0,1[}$;}
\Begin{$k\leftarrow 0$;

{Select $u_k \in \partial h(x_k)$ and compute the unique solution $y_k$ of
\begin{equation}\label{eq:sub_prob_BDCA}
\tag{$\mathcal{P}_{k}$} ~
\displaystyle\min_{x\in\R^m} \left\{\phi_k(x) := g(x) - \iprods{u_k, x}\right\};
\end{equation}}\\
$d_{k}\leftarrow y_{k}-x_{k}$;

\uIf{$\|d_{k}\|>\varepsilon$}{
Choose any $\overline{\lambda}_k\geq 0$ and set $\lambda_k\leftarrow\overline{\lambda}_k$; \\
\While{$\phi(y_{k}+\lambda_{k}d_{k})>\phi(y_{k})-\alpha\lambda_{k}^2\|d_{k}\|^{2}$}{$\lambda_{k}\leftarrow\beta\lambda_{k}$;}
$x_{k+1}\leftarrow y_{k}+\lambda_{k}d_{k}$;}
\Else{\textbf{stop} and \textbf{return} $y_k$;}
$k\leftarrow k+1$ and \textbf{go to} line~3;}
\end{algorithm}

Algorithmically, the BDCA is nothing more than the classical DCA with a line search procedure using an Armijo type rule. Note that the backtracking step in \Cref{alg:BDCA} (lines 6--9) terminates finitely thanks to~\Cref{fact:DCA}\Cref{fact:DCA_iii}.

We state next the basic convergence results for the sequences generated by BDCA (for more, see~\cite{BDCA2018,nBDCA}). Observe that DCA can be seen as a particular case of BDCA if one sets $\overline{\lambda}_k=0$, so the following result applies to both~\Cref{alg:DCA,alg:BDCA}.

\begin{fact}\label{fact:BDCA}
For any $x_0\in\R^m$, either \Cref{alg:BDCA} (BDCA) 
with $\varepsilon=0$ returns a critical point of~\eqref{eq:P}, or it generates an infinite sequence such that the following properties hold.
\begin{enumerate}[label=(\alph*)]
\item $\{\phi(x_k)\}$ is monotonically decreasing and convergent to some $\phi^\star$.
\item Any limit point of $\{x_k\}$ is a critical point of \eqref{eq:P}. In addition, if $\phi$ is coercive then there exists a subsequence of $\{x_k\}$ which converges to a critical point of \eqref{eq:P}.
\item It holds that $\sum_{k=0}^{+\infty}\norms{d_k}^2<+\infty$. Furthermore, if there is some $\overline{\lambda}$ such that $\lambda_k\leq\overline{\lambda}$ for all $k\geq 0$, then $\sum_{k=0}^{+\infty}\|x_{k+1}-x_{k}\|^{2}<+\infty$.
\end{enumerate}
\end{fact}
\begin{proof}
See~\cite[Theorem~3.6]{nBDCA}.
\end{proof}

\subsection{Positive Spanning Sets}

Most directional direct search methods are based on the use of positive spanning sets (see, e.g.,~\cite[Section~5.6.3]{book} and~\cite[Chapter~7]{DFO}). Let us recall this concept here.

\begin{definition}
We call \emph{positive span} of a set of vectors $\{v_1,v_2,\ldots,v_r\}\subset\R^m$ the convex cone generated by this set, i.e.,
$$\left\{v\in\R^m: v=\alpha_1v_1+\cdots+\alpha_rv_r, \quad  \alpha_i\geq 0,\, i=1,2,\ldots,r\right\}.$$
A set of vectors in $\R^m$ is said to be a \emph{positive spanning set} if its positive span is the whole space $\R^m$. A set $\{v_1, v_2, \ldots, v_r\}$ is said to be \emph{positively dependent} if one of the vectors is in the positive span generated by the remaining vectors; otherwise, the set is called \emph{positively independent}. A \emph{positive basis} in $\R^m$ is a positively independent set whose positive span is $\R^m$.
\end{definition}

Three well-known examples of positive spanning sets are given next.

\begin{example}[Positive basis]\label{ex:basis}
Let $e_1,e_2,\ldots,e_m$ be the unit vectors of the standard basis in $\R^m$. Then the following sets are positive basis in $\R^m$:
\begin{subequations}\label{eq:positive_basis}
 	\begin{align}
 	D_1 &:=\{\pm e_1,\pm e_2,\ldots,\pm e_m\}, \label{eq:D1} \\
 	D_2 &:=\left\{e_1,e_2,\ldots,e_m, -\textstyle\sum_{i=1}^m e_i \right\}, \label{eq:D2} \\
 	D_3  &:=\left\{ v_1,v_2,\ldots,v_m,v_{m+1}\in\R^m,\quad  \text{ with } \begin{array}{c}
 	 v_i^Tv_j=\frac{-1}{m}, \text{ if } i\neq j, \\
 	  \|v_i\|=1, \, i=1,2,\ldots,m+1.\end{array}\right\} . \label{eq:D3}
 	\end{align}
\end{subequations}
A possible construction for $D_3$ is given in~\cite[Corollary~2.6]{DFO}.
\end{example}

Recall that the BDCA provides critical points of~\eqref{eq:P} which are not necessarily d-stationary points (\Cref{fact:BDCA}). Theoretically, see~\cite[Section~3.3]{TT97}, if $x^\star$ is a critical point which is not d-stationary, one could restart BDCA by taking $x_0:=x^\star$ and choose $y_0\in\partial h(x_0)\setminus\{\nabla g(x_0)\}$. Nonetheless, observe that this is only applicable when the algorithm converges in a finite number of iterations to $x^\star$, which does not happen very often in practice (except for polyhedral DC problems, where even a global solution can be effectively computed if $h$ is a piecewise linear function with a reasonable small number of pieces, see~\cite[\S4.2]{TT97}). Because of that, our goal is to design a variant of BDCA that generates a sequence converging to a d-stationary point. The following key result, proved in~\cite[Theorem~3.1]{BH18}, asserts that using positive spanning sets one can escape from points which are not d-stationary. We include its short proof.


\begin{fact}\label{fact:stationary_nodescent_PSD}
Let $\{v_1,v_2,\ldots,v_r\}$ be a positive spanning set of $\R^m$. A point $x^\star\in\dom{\phi}$ is a d-stationary point of~\eqref{eq:P} if and only if
\begin{equation}\label{eq:no_descent_PSD}
\phi'(x^\star; v_i)\geq 0,\quad\text{for all } i=1,2,\ldots,r.
\end{equation}
\end{fact}
\begin{proof}
The direct implication is an immediate consequence of~\Cref{fact:stationary_nodescent}. For the reverse implication, pick any $x^\star\in\dom{\phi}$ verifying~\eqref{eq:no_descent_PSD} and choose any $d\in\R^m$. Since $\{v_1,v_2,\ldots,v_r\}$ is a positive spanning set, there are $\alpha_1,\alpha_2,\ldots,\alpha_r\geq 0$ such that
$$d=\alpha_1v_1+\alpha_2v_2+\cdots+\alpha_rv_r.$$
According to~\cite[Theorem~23.1]{RT}, we have that
$$h'(x^\star;d)\leq \alpha_1h'(x^\star;v_1)+\cdots+\alpha_rh'(x^\star;v_r).$$
Hence, we obtain
\begin{align*}
\phi'(x^\star;d) & =  g'(x^\star;d)-h'(x^\star;d)\\
& = \langle \nabla{g}(x^\star), \alpha_1v_1+\alpha_2v_2+\cdots+\alpha_rv_r \rangle-h'(x^\star;d) \\
& \geq \textstyle\sum_{i=1}^r \alpha_i \langle \nabla{g}(x^\star); v_i\rangle - \textstyle\sum_{i=1}^r \alpha_i h'(x^\star;v_i)\\
& = \textstyle\sum_{i=1}^r \alpha_i\phi'(x^\star;v_i)\geq 0.
\end{align*}
Since $d$ was arbitrarily chosen, then~\eqref{eq:no_descent} holds and $x^\star$ is a d-stationary point of~\eqref{eq:P}.
\end{proof}

\section{Forcing BDCA to converge to d-stationary points}\label{sec:BDCA}

In this section we propose a new variant of BDCA  to solve problem~\eqref{eq:P}, called BDCA$+$. The idea  is to combine BDCA with a basic DFO routine which uses positive spanning sets. The first scheme aims at achieving a fast minimization of the objective function $\phi$, while the second one is used to avoid converging to critical points for which there is at least a descent direction (i.e., they are not d-stationary points and, thus, they cannot be local minima). Let us make some comments about the new scheme BDCA$+$, which is stated in Algorithm~\ref{alg:BDCA2}.
\begin{itemize}
\item Subproblem~\eqref{eq:sub_prob_BDCA2} in line 3 corresponds to the classical DCA step for {solving}~\eqref{eq:P}.
\item Lines 5 to 10 encode the boosting line search step used in BDCA. If the current iterate is (numerically) not a critical point, then the algorithm performs a line search step at $y_k$ along the direction $d_k$ to improve the objective values of \eqref{eq:P}.
\item Line 11 to 19 correspond to a direct search DFO technique. It is run only when BDCA was stopped, in order to check if the point obtained is d-stationary. To this aim, it performs a backtracking search along each of the directions belonging to a positive spanning set $D$ of $\R^m$. If it reaches a point whose objective value is smaller, then we move to that point and run BDCA again from there. Otherwise, there is not descent direction in $D$ and, according to~\Cref{fact:stationary_nodescent_PSD}, the point we have found must be (numerically) d-stationary.
\item The choice $\overline{\lambda}_k=0$ for all $k$ is allowed, which corresponds to adding a direct search step to DCA.
\end{itemize}

\begin{algorithm}[ht!]\caption{BDCA$+$ (Boosted DC Algorithm combined with DFO)}\label{alg:BDCA2}
\KwIn{An initial point $x_{0}\in \R^m$, a positive spanning set $D$. Choose three nonnegative parameters $\varepsilon_1,\eta,\tau\geq 0$, three positive ones $\alpha,\varepsilon_2,\overline{\mu}>0$, and $\beta_1,\beta_2\in{]0,1[}$;}
\Begin{$k\leftarrow 0$, $\mu\leftarrow\overline{\mu}$;

Select $u_k \in \partial h(x_k)$ and compute the unique solution $y_k$ of
\begin{equation}\label{eq:sub_prob_BDCA2}
\tag{$\mathcal{P}_{k}$} ~
\displaystyle\min_{x\in\R^m} \left\{\phi_k(x) := g(x) - \iprods{u_k, x}\right\};
\end{equation}
$d_{k}\leftarrow y_{k}-x_{k}$;

\uIf{$\|d_{k}\|>\varepsilon_1$}{
Choose any $\overline{\lambda}_k\geq 0$ and set $\lambda_k\leftarrow\overline{\lambda}_k$; \\
\While{$\phi(y_{k}+\lambda_{k}d_{k})>\phi(y_{k})-\alpha\lambda_{k}^2\|d_{k}\|^{2}$}{$\lambda_{k}\leftarrow\beta_1\lambda_{k}$;}
$x_{k+1}\leftarrow y_{k}+\lambda_{k}d_{k}$;}
\Else{$\mu\leftarrow \eta\mu+\tau$; \\
\uIf{$\phi(y_{k}+\mu v)<\phi(y_{k})$ for some $v\in D$}{$x_{k+1}\leftarrow y_{k}+\mu v$;}
\uElseIf{$\mu>\varepsilon_2$}{$\mu\leftarrow\beta_2\mu$ and \textbf{go to} line~12;}
\Else{\textbf{stop} and \textbf{return} $y_k$;}}
$k\leftarrow k+1$ and \textbf{go to} line~3;}
\end{algorithm}

The following constructive example serves to illustrate the different behavior of DCA, BDCA and BDCA$+$.

\begin{example}[{{\cite[Example~3.3]{nBDCA}}}]\label{ex:ex1}
Consider the function $\phi:\R^2\to\R$ defined by
\begin{equation*}
\phi(x,y):=x^2+y^2+x+y-|x|-|y|.
\end{equation*}
Consider a corresponding DC decomposition $\phi=g-h$ of $\phi$ with
\begin{equation*}
g(x,y):=\frac{3}{2}(x^2+y^2)+x+y \quad\text{and}\quad h(x,y):=|x|+|y|+\frac{1}{2}(x^2+y^2).
\end{equation*}
Observe that $g$ and $h$ satisfy \Cref{as:A1,as:A2,as:A3}. It can be easily checked that $\phi$ has four critical points, namely $(0,0)$, $(-1,0)$, $(0,-1)$ and $(-1,-1)$, of which only the latter is a d-stationary point (and also the global minimum).

In~\Cref{fig:example} we show the iterations generated by DCA (\Cref{alg:DCA}) and BDCA$+$ (\Cref{alg:BDCA2}) from the same starting point $x_0=(0,1)$. The DCA converges to the critical point $(0,0)$. The BDCA escapes from this point but still gets stuck at $(0,-1)$, which is also a critical point which is not d-stationary. After applying once the DFO scheme (dashed line), we observe that BDCA successfully converges to the d-stationary point $(-1,-1)$, which is in fact the global minimum of the problem.

\begin{figure}
\centering
\includegraphics[width=0.6\textwidth]{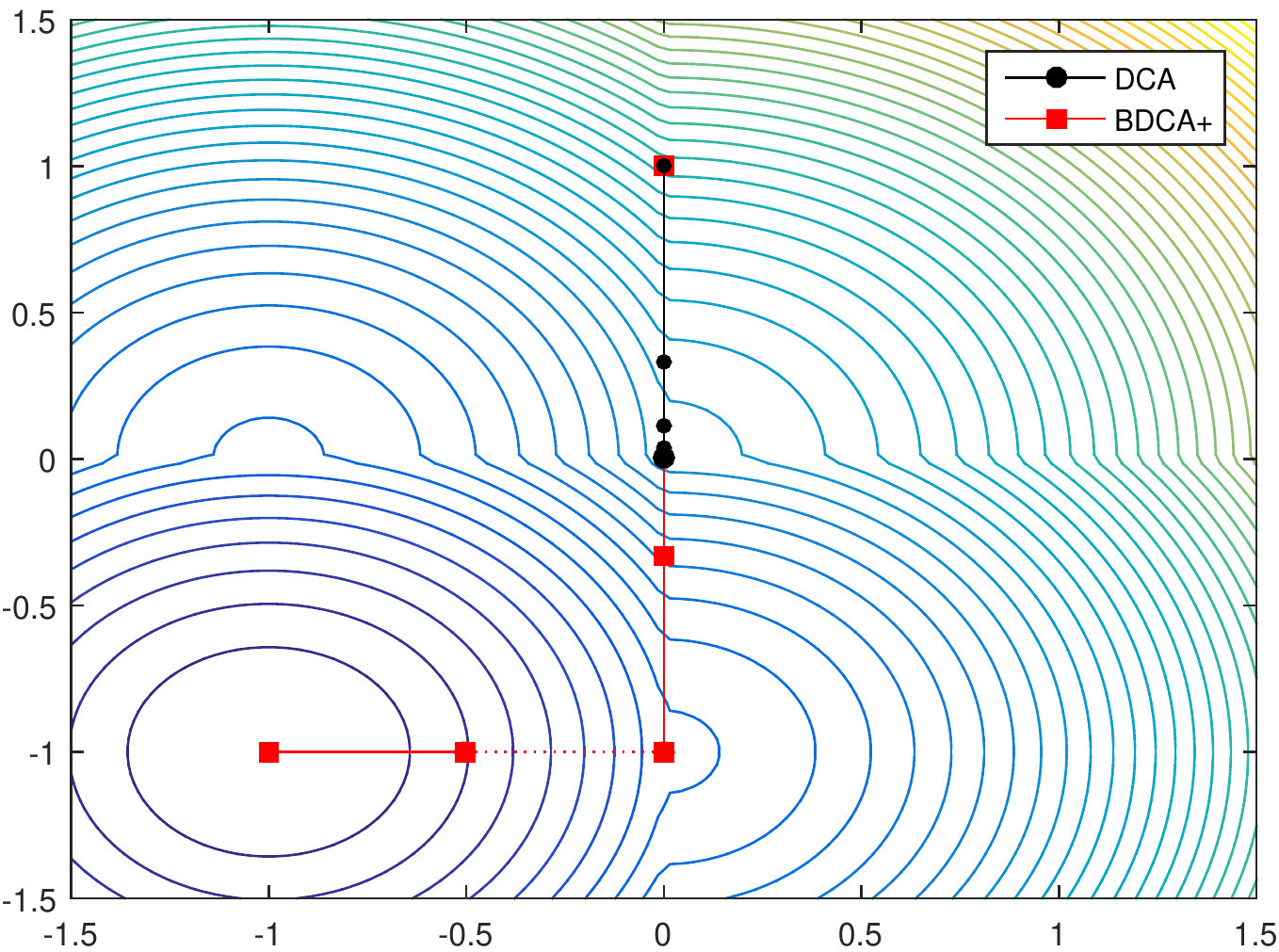}
\caption{Ilustration of~\Cref{ex:ex1}.}\label{fig:example}
\end{figure}

To demonstrate the advantage of BDCA$+$ we compute the number of instances, out of one million random starting points uniformly distributed in $[-1.5,1.5]\times[-1.5,1.5]$, that each algorithm has converged to each of the four critical points. The results are summarized in \Cref{tbl:example}.

\begin{table}[ht!]\centering
\begin{tabular}{|c|c|c|c|c|}
\cline{2-5}
\multicolumn{1}{c|}{}
& $(-1,-1)$ & $(-1,0)$ & $(0,-1)$ & $(0,0)$\\
\hline
DCA     & 249,821  & 250,671 & 249,944     & 249,564\\
BDCA    & 996,221  &   1,897 &   1,882     &   0\\
BDCA$+$ & 1,000,000&  0      &  0     &   0\\
\hline
\end{tabular}\caption{For one million random starting points in $[-1.5,1.5]\times[-1.5,1.5]$, we count the sequences generated by DCA, BDCA and BDCA$+$ converging to each of the four d-stationary points}\label{tbl:example}
\end{table}

From Table~\ref{tbl:example}, we observe that DCA converged to each of the four critical points with the same probability, while BDCA converged to the global minimum in $99.6\%$ of the instances. The best results where obtained by BDCA$+$, which always converged to the global minimum $(-1,-1)$.
\end{example}


\section{Numerical experiments}\label{sec:application1}
In this section, we provide the results of some numerical tests to compare the performance of BDCA$+$ (\Cref{alg:BDCA2}) and the classical DCA (\Cref{alg:DCA}). To this aim we turn to the same challenging clustering problem tested in~\cite[Section~5.1]{nBDCA}, where both algorithms have troubles for finding good solutions due to an abundance of critical points. A different algorithm based on the DC programming approach to solve this problem was introduced in~\cite{BTU16}. This algorithm is also proved to converge to d-stationary~points.

All the codes were written in Python 2.7 and the tests were run on a desktop of Intel Core i7-4770 CPU 3.40GHz with 32GB RAM, under Windows 10 (64-bit). The following strategies have been followed in all the experiments:

\begin{itemize}
\item The trial step size $\overline{\lambda}_k$ in the boosting step of BDCA (line 6 in \Cref{alg:BDCA,alg:BDCA2}) was chosen to be self-adaptive, as in~\cite[Section~5]{nBDCA}, which proceeds as follows:
\begin{enumerate}[label=\arabic*.]
\item Set $\overline{\lambda}_0=0$ and fix any $\gamma>1$.
\item Choose any $\overline{\lambda}_1>0$ and obtain $\lambda_1$ by backtracking.
\item For $k\geq 2$,
\begin{align*}
& \textbf{if }  (\lambda_{k-2}=\overline{\lambda}_{k-2} \textbf{ and } \lambda_{k-1}=\overline{\lambda}_{k-1}) \textbf{ then}\\
& \qquad \text{set } \overline{\lambda}_k:=\gamma\lambda_{k-1}; \\
& \textbf{else }  \text{set }\overline{\lambda}_k:=\lambda_{k-1};
\end{align*}
and obtain $\lambda_k$ by backtracking.
\end{enumerate}

\item In our numerical tests we observed that the accepted step sizes $\mu$ in the DFO step of~\Cref{alg:BDCA2} usually decrease (nearly always).
For this reason,  we used $\eta:=\frac{1}{\beta_2}$ and $\tau:=\varepsilon_2$ in the choice of the initial value of $\mu$ at line 12 in~\Cref{alg:BDCA2}.
By this way, we allow a slight increase in the value of the step size with respect to the previous one, while  we can avoid wasting too much time in this backtracking.

\item We tested the three positive basis presented in~\Cref{ex:basis}. Surprisingly, the basis with equally spaced angles $D_3$ in~\eqref{eq:D3} performs worse than the others in our test problem. In fact, the best choice was the basis $D_1$ in~\eqref{eq:D1}, and this is the one we have employed in all the experiments throughout this section.
\item We used the parameter setting as $\alpha:=0.0001$, $\varepsilon_1:=10^{-8}$, $\varepsilon_2:=10^{-4}$, $\overline{\mu}:=10$, $\gamma=2$, $\overline{\lambda}_1:=10$, $\beta_1:=0.25$ and $\beta_2:=0.5$.
\end{itemize}

\paragraph{The Minimum Sum-of-Squares Clustering Problem:} Given a collection of $n$ points, $\{a^1,a^2,\ldots,a^n\in\R^m\}$, the goal of \emph{clustering} is to group them in $k$ disjoint sets (called \emph{clusters}), $\{A^1,A^2,\ldots,A^k\}$, under an optimal criterion. For each cluster $A_j$, $j=1,2,\ldots,k$, consider its centroid $x^j$ as a representative. The \emph{Minimum Sum-of-Squares Clustering} criterion asks for the configuration that minimizes the sum of squared distances of each point to its closest centroid, i.e. the solution to the optimization problem
\begin{equation}\label{eq:MSOS}
\displaystyle\min_{x^1,\ldots,x^k\in\R^{m}} \left\{\varphi(x^1,\ldots,x^k):= \frac{1}{n}\sum_{i=1}^n\min_{j=1,\ldots,k}\norms{x^j-a^i}^2\right\}.
\end{equation}
We can rewrite the objective in~\eqref{eq:MSOS} as a DC function (see~\cite{BTU16,CYY18,OB15}) with
\begin{align*}
g(x^1,\ldots,x^k)&:=\frac{1}{n}\sum_{i=1}^n\sum_{j=1}^k\norms{x^j-a^i}^2+\frac{\rho}{2}\sum_{j=1}^k\norms{x^j}^2,\\
h(x^1,\ldots,x^k)&:=\frac{1}{n}\sum_{i=1}^n\max_{j=1,\ldots,k}\sum_{t=1,t\neq j}^k\norms{x^t-a^i}^2+\frac{\rho}{2}\sum_{j=1}^k\norms{x^j}^2;
\end{align*}
where $g$ and $h$ satisfy \Cref{as:A1,as:A2,as:A3} for all $\rho>0$ (in our tests, we took $\rho=\frac{1}{nk}$).

\paragraph{Data Set and Experiments:} Our data set is the same one considered in~\cite{nBDCA}, which consists of the location of $4001$ Spanish cities in the peninsula with more than 500 inhabitants\footnote{The data can be retrieved from the Spanish National Center of Geographic Information at \href{http://centrodedescargas.cnig.es}{http://centrodedescargas.cnig.es}.}. In \Cref{fig:map} we compare the iterations generated by DCA and BDCA$+$ for finding a partition into $20$ clusters from the same random starting point $x_0\in\R^{2\times 20}$ (marked with a black cross). We observe that DCA converges to a critical point which is far from being optimal, as there are three clusters without any cities assigned. On the other hand, although BDCA apparently converges to the same critical point, the DFO step allows BDCA$+$ to escape from points which are not d-stationary and reach a better solution.

\begin{figure}[ht!]
\centering
\subfigure[DCA, objective value obtained: $0.4778$]{\includegraphics[width=0.662\textwidth]{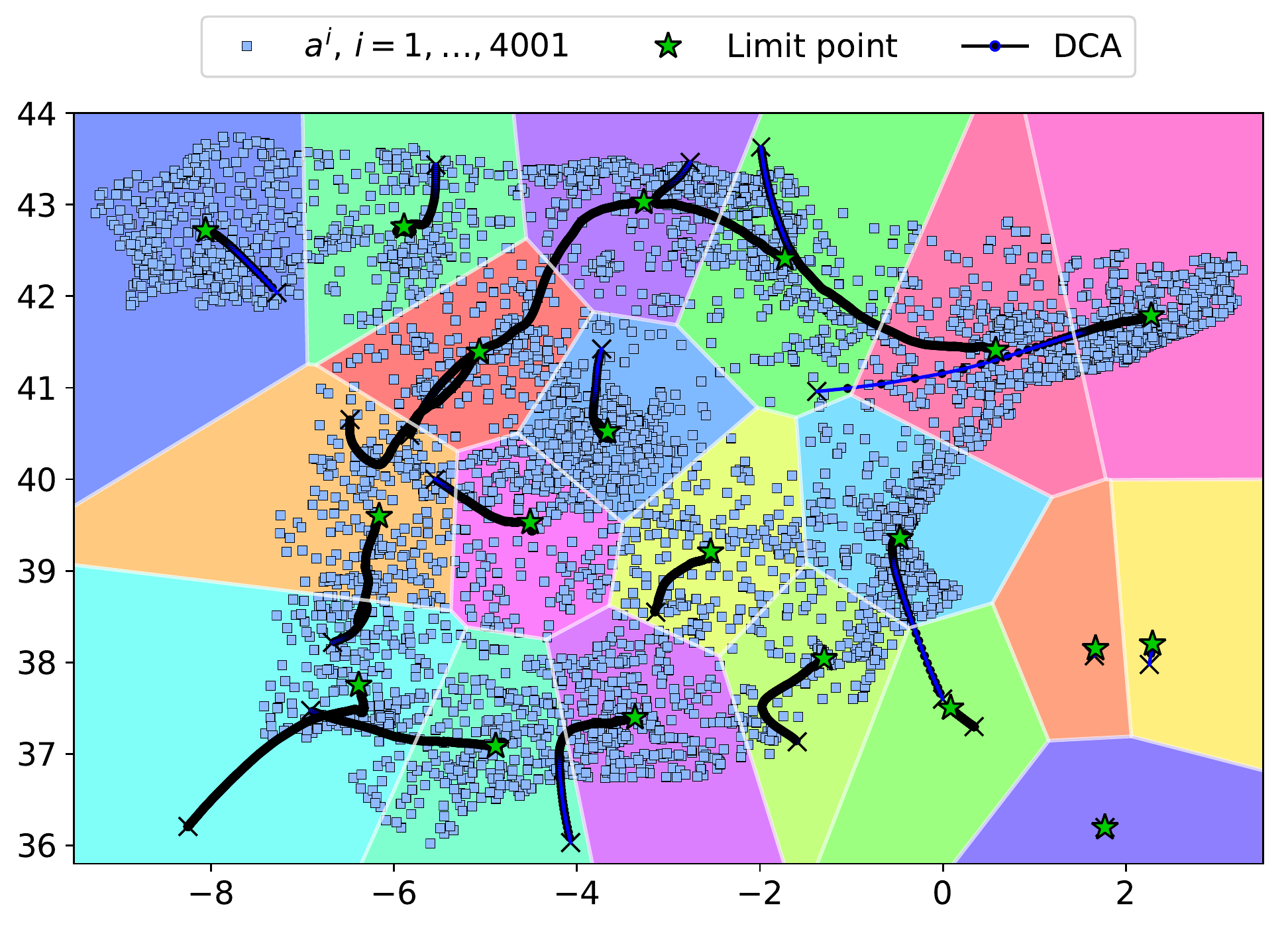}}
\subfigure[BDCA$+$, objective value obtained: $0.4076$]{\includegraphics[width=0.662\textwidth]{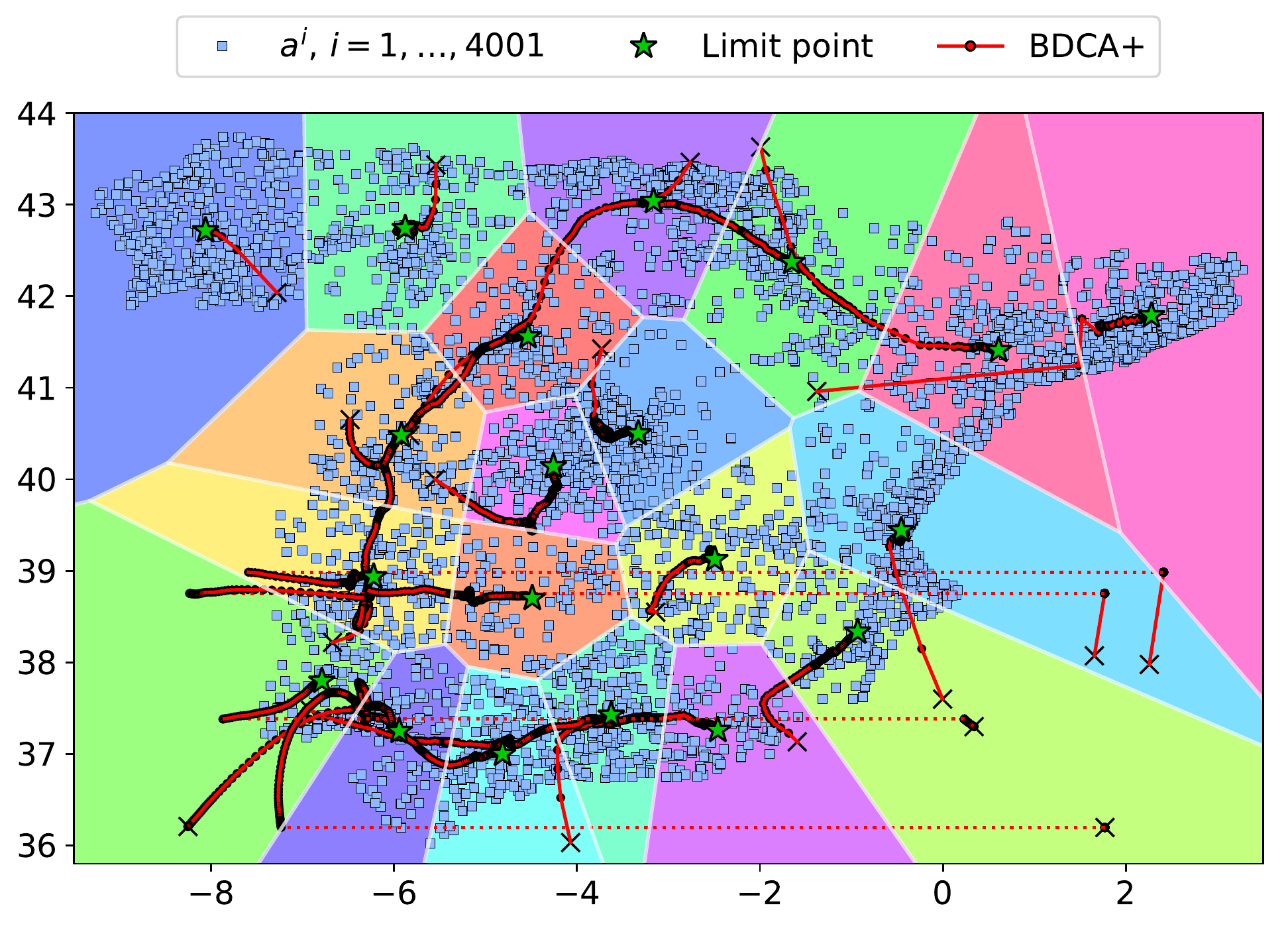}}
\caption{Iterations and limit points generated by DCA and BDCA$+$ for grouping the Spanish cities in the peninsula into 20 clusters from the same random starting point. The DFO step in line 14 of \Cref{alg:BDCA2} was run 10 times (these steps are marked with a dashed line).}\label{fig:map}
\end{figure}


To corroborate these results, we repeated the experiment for different number of clusters $k\in\{20,40,60,80\}$. For each of these values, we run DCA and BDCA$+$ from $50$ random starting points. The results are shown in~\Cref{fig:exp_spain}, where we can clearly observe that BDCA$+$ outperforms DCA, not only in terms of the objective value attained, but even in running time. Observe that it is not really fair to compare the running time of DCA and BDCA$+$, because DCA simply stops at a critical point without incorporating the time-consuming DFO step that guarantees d-stationarity. Nonetheless, the speedup obtained by the line search of BDCA allows BDCA$+$ to still converge faster than DCA in most of the instances. As expected, BDCA$+$ becomes slower as the size of the problem increases, due to the DFO step. Despite that, notice that for $80$ clusters the best solution provided by DCA among the $50$ instances is still worse than the worst solution obtained by BDCA$+$. That is, any of the runs of BDCA$+$ was able to obtain a better solution than $50$ restarts of DCA.

\begin{figure}[ht!]
\subfigure{\includegraphics[width=0.5\textwidth]{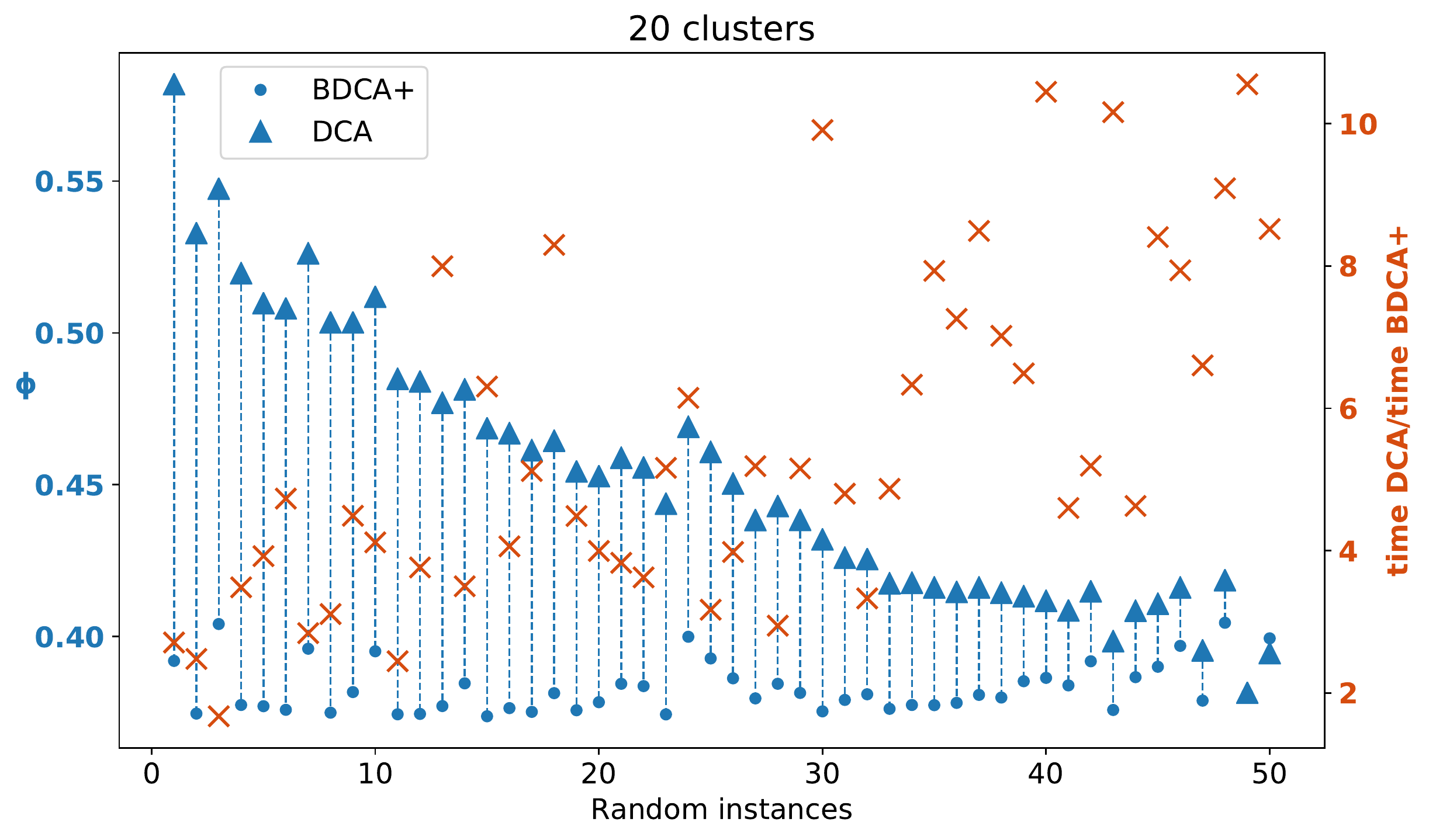}}
\subfigure{\includegraphics[width=0.5\textwidth]{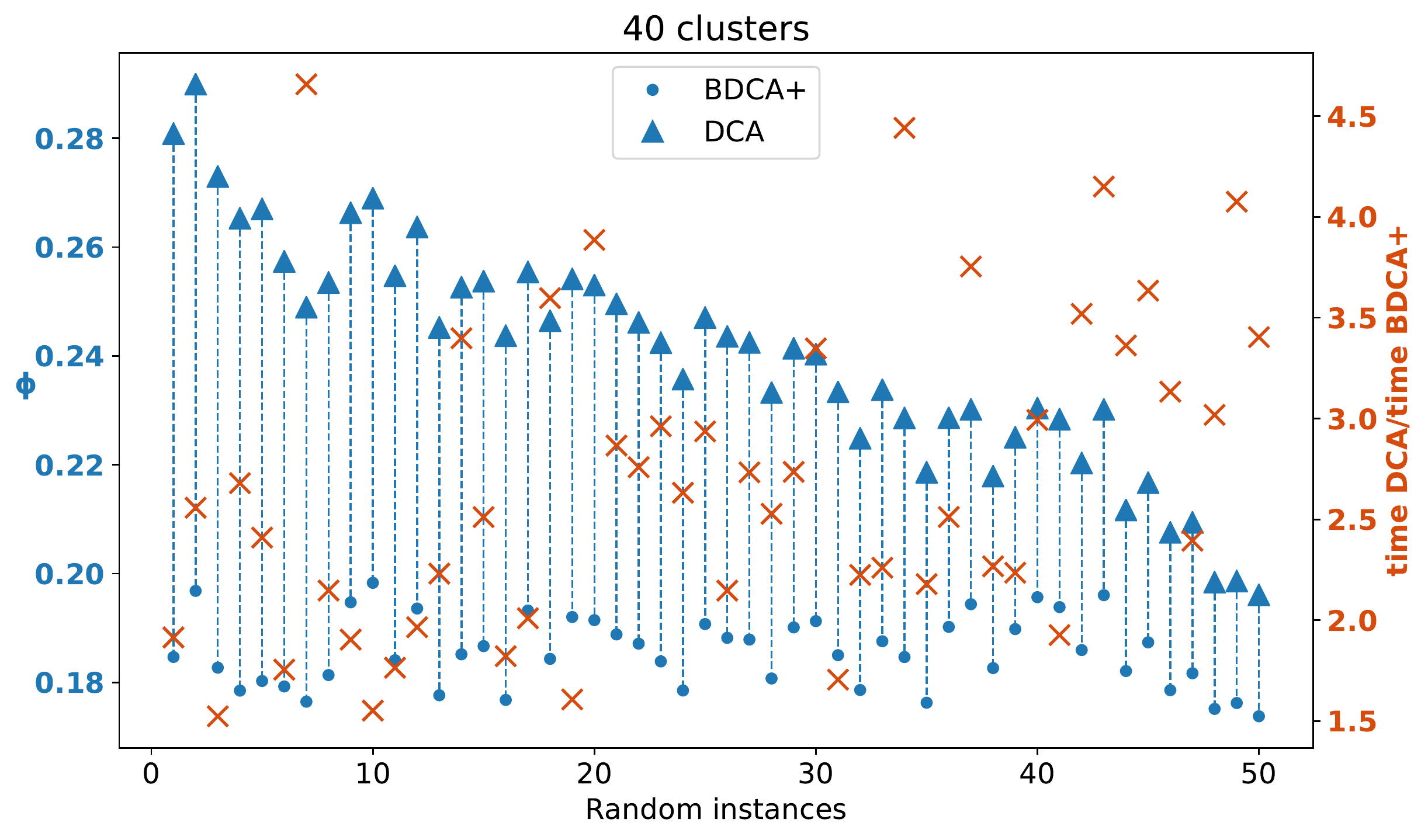}}
\subfigure{\includegraphics[width=0.5\textwidth]{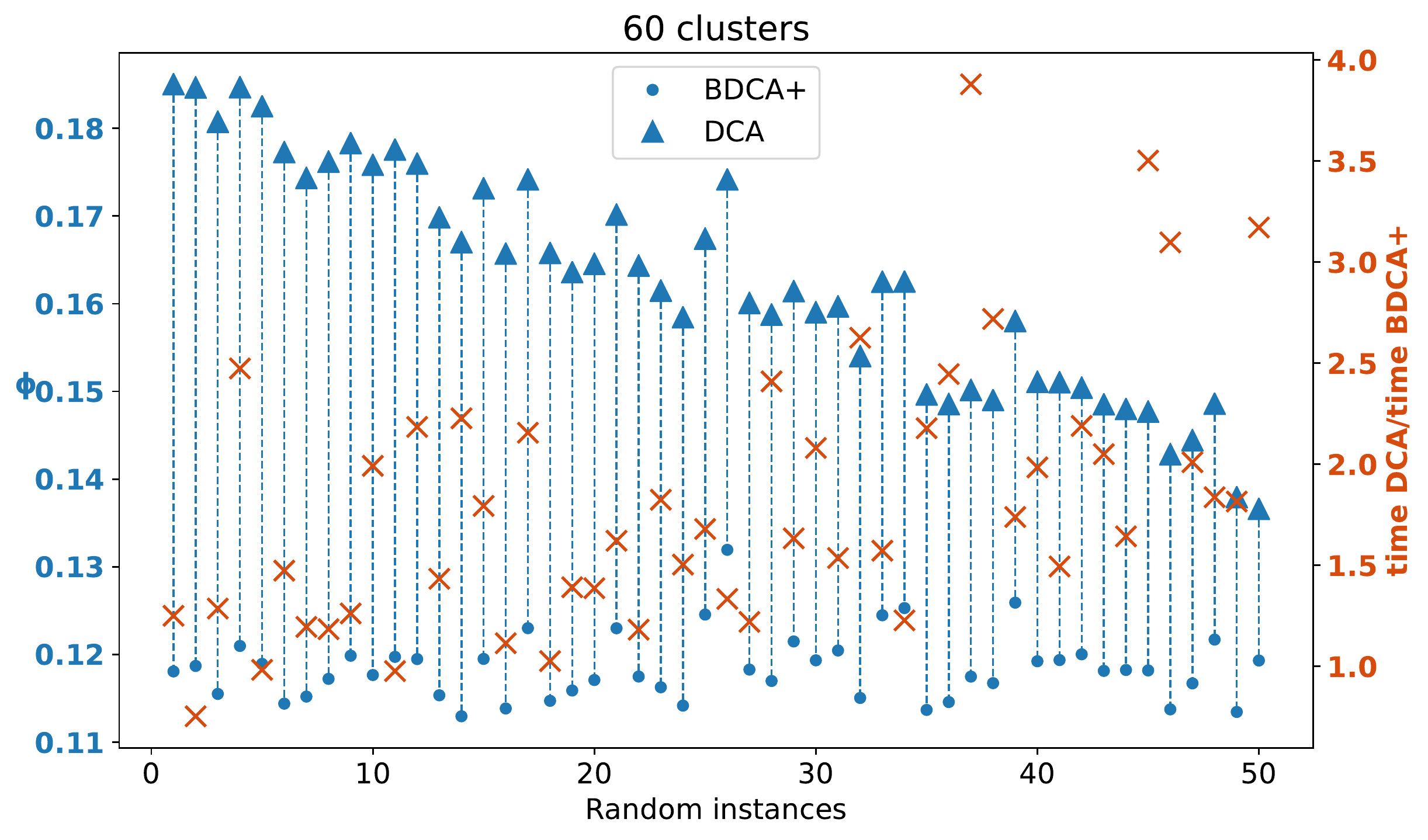}}
\subfigure{\includegraphics[width=0.5\textwidth]{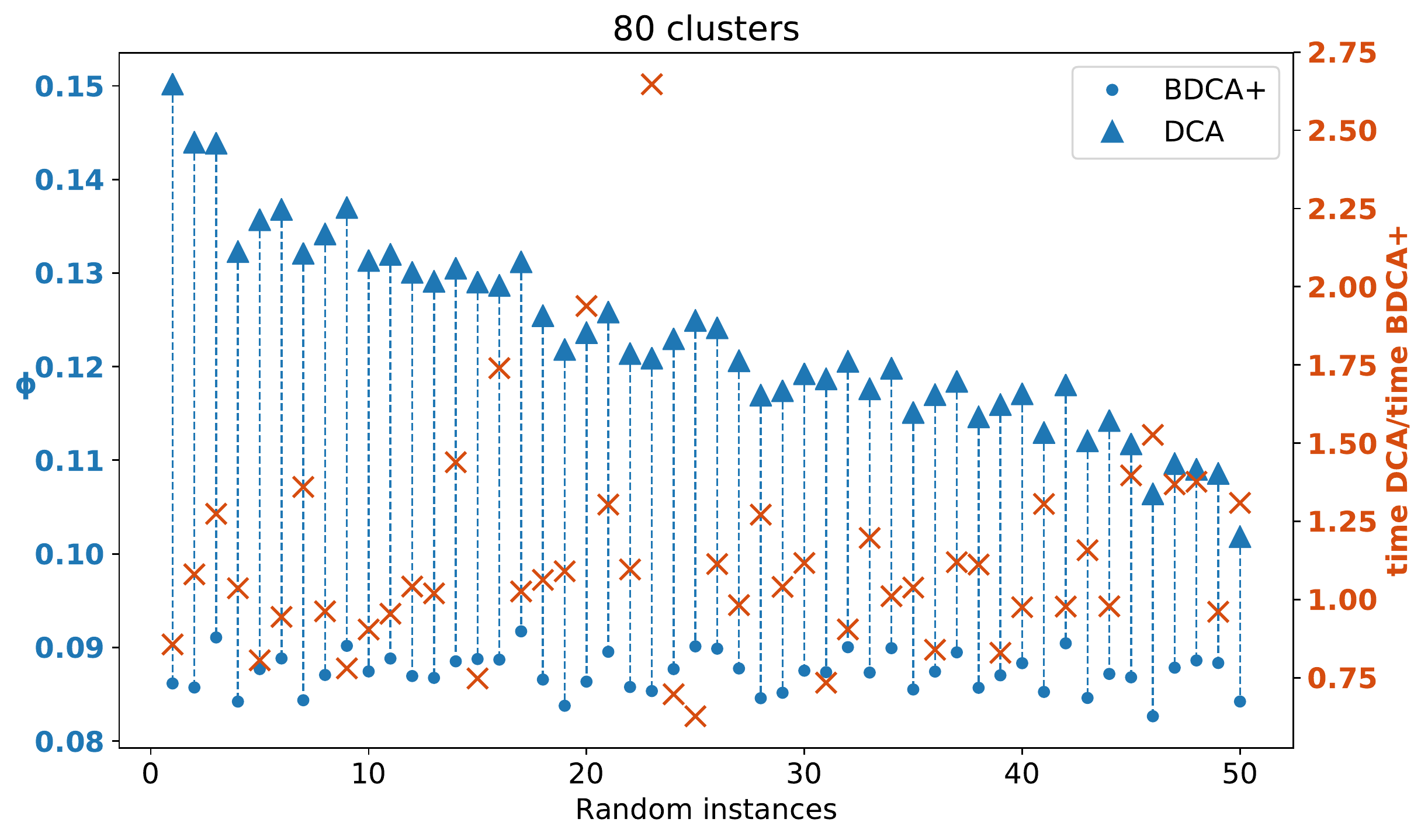}}
\caption{Comparison between the DCA and the BDCA$+$ for classifying the Spanish cities in the peninsula into $k$ clusters for $k\in\{20,40,60,80\}$. For each of these values, both algorithms were run from $50$ random starting points. We represent the objective value achieved in the limit point by each algorithm (left axis, in blue), as well as the ratio between the CPU time required by DCA with respect to the one needed by BDCA$+$ (right axis, orange crosses). Instances were sorted on the x-axis in descending order according to the gap between the objective values at the limit points found by the algorithms.}\label{fig:exp_spain}
\end{figure}



\section{Concluding remarks}\label{sec:concl}
We have proposed a combination between the Boosted DC Algorithm (BDCA) and a simple direct search Derivative-Free Optimization (DFO) technique for minimizing the difference of two convex functions, the first of which is assumed to be smooth. The BDCA is used for minimizing the objective function, while the DFO step permits to force the iteration to converge to d-stationary points (i.e. to points where there exists no descent direction), rather than just critical points.

The good behavior of the new algorithm, called BDCA$+$, has been demonstrated by numerical experiments in a clustering problem. The new scheme generates better solutions than the classical DCA in nearly all the instances tested. Moreover, this improvement in the quality of the solutions has not caused an important loss in the time spent by the algorithm. In fact, BDCA$+$ was faster than DCA in most of the cases, thanks to the large acceleration achieved by the line search boosting step of BDCA.

\paragraph{Acknowledgements}
We thank the referees for their constructive comments which helped us improve the presentation of the paper.
The first author was supported by MINECO of Spain and ERDF of EU, as part of the Ram\'on y Cajal program (RYC-2013-13327) and the grants MTM2014-59179-C2-1-P and PGC2018-097960-B-C22.
The second author was supported by MINECO of Spain and ESF of EU under the program ``Ayudas para contratos predoctorales para la formaci\'on de doctores 2015'' (BES-2015-073360).
The last author was supported by FWF (Austrian Science Fund), project M2499-N32 and Vietnam National Foundation for
Science and Technology Development (NAFOSTED) project 101.01-2019.320.


\end{document}